\documentclass[journal]{IEEEtran}

\usepackage{moreverb,url}
\usepackage[colorlinks,bookmarksopen,bookmarksnumbered,citecolor=red,urlcolor=red]{hyperref}
\usepackage{graphics} 
\usepackage{epsfig} 
\usepackage{amsmath} 
\usepackage{amsthm} 
\usepackage{amssymb}  
\usepackage{fourier} 
\usepackage{subcaption}
\usepackage{cite}
\usepackage[lined,linesnumbered,ruled,vlined]{algorithm2e} 
\usepackage{times}
\usepackage{multirow} 

\newtheorem{theorem}{Theorem}

\newtheorem{lemma}{Lemma}

\newtheorem{Def}{Definition}
\newtheorem{Ass}{Assumption}
\newcommand{\vect}[1]{\boldsymbol{#1}}

\graphicspath{{figures/}}

\begin{document}
%
\title{Finite-time enclosing control for multiple moving targets: a continuous estimator approach}

%

\author{Liang~Zhang 
\thanks{Liang Zhang is with the School of Electrical Engineering and Automation, Anhui university, Hefei 23000, China. 
{\tt\small liangzhang@ahu.edu.cn}}
}

\maketitle
 
\begin{abstract}
This work addresses the finite-time enclosing control problem where a set of followers are deployed to encircle and rotate around multiple moving targets with a predefined spacing pattern in finite time. A novel distributed and continuous estimator is firstly proposed to track the geometric center of targets in finite time using only local information for every follower. Then a pair of decentralized control laws for both the relative distance and included angle, respectively, are designed to achieve the desired spacing pattern in finite time based on the output of the proposed estimator. 
Through both theoretical analysis and simulation validation, we show that the proposed estimator is continuous and therefore can avoid dithering control output while still inheriting the merit of finite-time convergence. The steady errors of the estimator and the enclosing controller are guaranteed to converge to some bounded and adjustable regions around zero. 
\end{abstract}


%
\IEEEpeerreviewmaketitle

\section{Introduction}
%
%
%
%
\IEEEPARstart{E}{ncircling} formation has been found very popular in natural swarms for its potential advantages in predation and protection \cite{beni2004swarm,cushing1968fish,berlinger2021implicit,dorigo2020reflections}, which draws more and more attention to the surrounding or enclosing control problems in recent years. The objectives of these problems usually involve attaining a formation for a group of followers to orbit around a common point with a predefined spacing pattern such that single or multiple targets (or leaders) can always be contained within the formation.

Early researches mainly concentrate on enclosing a single target due to its potential applications of attacking, entrapping, or protecting a target object \cite{yamaguchi1999cooperative,schumacher2005ground}. Results on the single-static-target enclosing problem can be found in \cite{lan2010distributed,zheng2015enclosing,li2018cooperative}. In \cite{lan2010distributed}, a balanced enclosing pattern with uniform circular formation around the target at equal angular distances is studied for the unicycle-type mobile followers. Assuming only bearing measurements are available to the followers, the control methods to enclose a stationary target are investigated in \cite{zheng2015enclosing}. Moving on, the enclosing control has been extended to nonlinear Euler-Lagrange dynamics with collision avoidance in \cite{li2018cooperative}. 

More works attempt to enclose a single moving target. When the followers are assumed to have access to the target's velocity, The cyclic pursuit strategy is developed where each agent simply pursues its predecessor to achieve the target-capturing task in 3D space, for holonomic agents in \cite{kim2007cooperative} and more general MIMO agents in \cite{hara2008distributed}, respectively. Later on, a local information control law is proposed in \cite{GUO2010Local} for the moving-target-enclosing problem where the velocity of the target is bounded but not known a priori. More recent works on the moving-target-enclosing problem are developed regarding the linear dynamics of agents with input saturation in \cite{XU2020Moving}, regarding the nonholonomic agents using feedback linearization control method in \cite{Peng2021Cooperative}, regarding mobile target with time-varying velocity under the consideration of collision avoidance in \cite{Dou2021Moving} and so on.

In the single-target case, the center of circular formation can be directly fixed on the target as its position is generally assumed to be observed by followers. However, enclosing multiple targets is more challenging because the followers must rotate around the targets' geometrical center, and in practice, the followers can hardly make direct observations of this center. Instead, an extra estimator should be maintained by every follower to estimate the targets' geometrical center using only local information, for example, the local measurements to a subset of targets and the exchanged messages from its neighbors. Enclosing multiple static leaders is investigated in \cite{chen2010surrounding,zhang2017surrounding} for both the first- and second-order followers. The designed estimators can exponentially converge to the geometrical center using only their initial measurements to the targets. Then traditional enclosing control approaches for enclosing a single target case can be directly applied to multiple targets case given the output of estimators. 

However, the exponentially convergent estimator can not be applied to enclosing multiple moving targets because the accumulated tracking error between the output of the estimator and the time-varying geometrical center will lead to fatal divergence on the enclosing controller, which results in the development of a finite-time convergent estimator in \cite{shi2015cooperative,sharghi2019finite,zhang2019finite,ma2017finite}. Such an estimator can ensure to exactly track the targets' geometrical center in finite time such that the accumulated error is tolerable by the enclosing controller. However, it currently still suffers from the dithering phenomenon (see Figure 3 in \cite{shi2015cooperative}) when either the estimator is converging or the adjacent estimators are similar due to the presence of discontinuous term. Although a smoothing function is suggested in \cite{shi2015cooperative} to relieve such a problem, no further evidence has been provided to guarantee that the smoothed estimator is still of finite-time convergence and its impact on enclosing control
also waits for further investigation. Besides, current finite-time enclosing controllers in \cite{sharghi2019finite,zhang2019finite,ma2017finite} mostly only consider the balanced spacing pattern and can only guarantee the finite-time convergence on the relative distance channel, while the included angle is still exponentially convergent.  

Regarding the aforementioned problems, this paper addresses the finite-time enclosing control problem for multiple moving targets. A continuous estimator is designed to exactly track the targets' geometrical center without dithering. Rigorous derivations are presented to show that the tracking error can be stabilized into a bounded and controllable stable region in finite time. Hereafter, two finite-time enclosing control laws are developed based on the output of the continuous estimator, which drive both the relative distance and the included angle to the predefined spacing pattern. Therefore, the steady enclosing formation is not necessarily a balanced one as considered in \cite{shi2015cooperative,sharghi2019finite}.

The contributions of the proposed method are mainly threefold, 
\begin{enumerate}
	\item A continuous estimator for the track of the targets' geometrical center, which doesn't cause dithering phenomenon and is still of finite-time convergence.  
	\item A pair of finite-time enclosing controllers, driving the errors of both relative distance and included angle to be stabilized into a small region in finite time. The steady spacing pattern can be arbitrarily specified and therefore the balanced cases in \cite{shi2015cooperative,sharghi2019finite} are included.  
	\item The principles for tuning the parameters in controller and estimator to adjust the tracking and controlling accuracy (i.e. the size of bounded stable regions)
\end{enumerate}

\section{Preliminaries}

Generally, there are two methods to prove the finite-time stability of a nonlinear system. The first one is by applying the following theorem \cite{wang2010finite},

\begin{theorem}
	Consider the nonlinear system $\dot{\vect{x}} = f(\vect{x})$ with $f(\vect{0}) = \vect{0}$, suppose that there exist a $C^1$ function $V(\vect{x}_0)$ defined on a neighborhood of the origin and two real numbers $c>0,\alpha \in (0,1)$, such that $$\dot{V}(\vect{x}) \le -c V^{\alpha}(\vect{x}),$$ then the origin of the system is finite-time stable and the upper bound of the settling time is $$T\le \frac{V^{1-\alpha}(\vect{0})}{c(1-\alpha)}.$$ In addition, given a real constant $\gamma < V(\vect{0})$, then the settling time $T_1$ guaranteeing $V(t)<\gamma, \forall t > T_1$ is $$T_1 \le \frac{V^{1-\alpha}(\vect{0}) - \gamma^{1-\alpha}}{c(1-\alpha)}.$$
\end{theorem}

Another method utilizes the merit of homogeneous system, which can be defined by the following definitions,

\begin{Def}
	A function $V:\mathbb{R}^n \rightarrow \mathbb{R} $ is homogeneous of degree $l$ with respect to the "standard dilation":
	\begin{equation}
		\Delta_{\lambda} (x_1,x_2,\cdots,x_n) = (\lambda x_1, \lambda x_2, \cdots, \lambda x_n)
	\end{equation}
	if and only if:
	$$V(\lambda x_1, \lambda x_2, \cdots, \lambda x_n) = \lambda^l V(x_1, x_2, \cdots, x_n)$$
\end{Def}

\begin{Def}
Considering the following system $\dot{\vect{x}} = f(\vect{x}),\vect{x} \in \mathbb{R}^n$, it is called to be homogeneous of degree $q$ with respect to the standard dilation if and only if the $i$-th component $f_i$ is homogeneous of degree $q+1$ with respect to the standard dilation, i.e. $f_i(\lambda x_1, \lambda x_2, \cdots ,\lambda x_n) = \lambda^{q+1}f_i(x_1,\cdots,x_n)$,  $\lambda >0, i =\{1,2,\cdots, n\}$.	
\end{Def}

Then the following theorem \cite{2005Geometric} can guarantee the finite-time stability of the homogeneous system, 

\begin{theorem} \label{Th. HomogeneousFiniteTimeStable}
	Let $f(\vect{x}), \vect{x} \in \mathbb{R}^n$ be a homogeneous vector field of degree $q$ with respect to the standard dilation, then the nonlinear system $\dot{\vect{x}} = f(\vect{x}),\vect{x} \in \mathbb{R}^n$ is finite-time stable if and only if it is asymptotically stable and $q<0$.
\end{theorem}

In addition, the following two lemmas are instrumental to simplify the derivations of our main results,

\begin{lemma} \label{lemma. SummationFraction}
	Given a vector $\vect{x} \in \mathbb{R}^{n}$ and a positive $0<p<1$, then we have $\left( \sum_{i=1}^{n} |x_i| \right)^p \le \sum_{i=1}^{n} |x_i|^p \le n^{1-p} \left( \sum_{i=1}^n |x_i| \right)^p $
\end{lemma}

\begin{lemma} \label{lemma. sum_sum}
	Given two sets of functions $\{ b_1(t),b_2(t),\cdots, b_n(t) \}$, $\{ c_1(t),c_2(t),\cdots, c_n(t) \}$ and a symmetric matrix $\vect{A} =[a_{ij}] \in \mathbb{R}^{n \times n}$, then
	$$\begin{aligned}
		&\sum_{i=1}^n \sum_{j=1}^n \left[ a_{ij} c_i(t) \left[ b_i(t) - b_j(t) \right]^{\frac{q}{p}} \right] \\
		& = \frac{1}{2} \sum_{i=1}^n \sum_{j=1}^n \left \lbrace a_{ij} \left[ c_i(t) - c_j(t) \right]\left[ b_i(t) - b_j(t) \right]^{\frac{q}{p}}    \right \rbrace 
	\end{aligned}$$
	where $p,q>0$ are two odd positives. In addition, if $c_i(t) = b_i(t), \forall i \in \{ 1,2,\cdots,n \}$, then
		$$\begin{aligned}
		&\sum_{i=1}^n \sum_{j=1}^n \left[ a_{ij} c_i(t) \left[ b_i(t) - b_j(t) \right]^{\frac{q}{p}} \right] \\
		& = \frac{1}{2} \sum_{i=1}^n \sum_{j=1}^n \left \lbrace a_{ij} \left[ b_i(t) - b_j(t) \right]^{1+\frac{q}{p}}    \right \rbrace 
	\end{aligned}$$
	And especially when $c_i(t) = 1, \forall i \in \{ 1,2,\cdots,n \}$, 
	$$\begin{aligned}
		\sum_{i=1}^n \sum_{j=1}^n \left[ a_{ij} \left[ b_i(t) - b_j(t) \right]^{\frac{q}{p}} \right]  = 0
	\end{aligned}$$
\end{lemma}

\section{Problem formulation}

Consider a group of moving targets $i, i \in  \mathcal{V}_L = \{ 1,2,\cdots, m \} $, called the leaders, that can move freely in a 2-D plane. Their dynamics can be formulated  by, 
\begin{equation} \label{eq. leadersDynamics}
	\vect{\dot{p}}_i^t = \vect{v}_i, \forall i \in \mathcal{V}_L
\end{equation}
In the same plane exists another group of moving agents, $i, i \in \mathcal{V}_F = \{ 1,2,\cdots, n \}$, called the followers, that are deployed to enclose the leaders with a predefined spacing pattern. The followers are also governed by
\begin{equation} \label{eq. followersDynamics}
	\vect{\dot{p}}_i=\vect{u}_i,\forall i \in \mathcal{V}_F
\end{equation}
where $\vect{p}_i^t = [p_{i,x}^t, p_{i,y}^t]^T, \vect{{p}}_i = [p_{i,x}, p_{i,y}]^T$ are the positions of leaders and followers, respectively. $\vect{v}_i \in \mathbb{R}^2$ is the velocity of leader $i$ and $\vect{u}_i \in \mathbb{R}^2$ is the control input of follower $i$. We firstly have the following assumptions regarding the communication and observation topology,
\begin{Ass} \label{ass. undirectNetwork}
	The communication topology among followers is undirect and connected. The communication weights are either $0$ or $1$.
\end{Ass}
\begin{Ass} \label{ass. leastObservation}
	Each leader is assumed to be observed by at least one follower during the enclosing. The follower can only observe its neighboring leaders' positions while the velocities are unknown.
\end{Ass}
In addition, we further assume that, 
\begin{Ass} \label{ass. boundedVelocityLeader}
	There exists a known positive constant $\beta$, such that $$||\vect{v}_i||_2 \le \beta, \forall i \in \mathcal{V}_L.$$
\end{Ass} 


\begin{figure}
	\centering
	\includegraphics[width=0.8\linewidth]{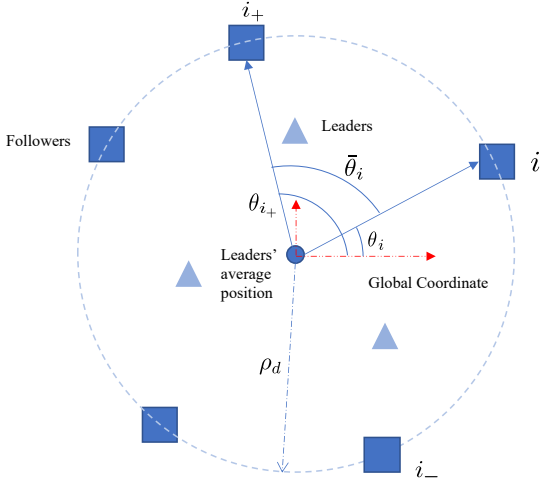}
	\caption{Configuration for the enclosing problem}
	\label{fig. enclosingDemonstration}
	\vspace{-10pt}
\end{figure}

The enclosing problem is demonstrated in Fig.~\ref{fig. enclosingDemonstration}, where the neighbor strategy taken in this work is the same as that in \cite{GUO2010Local,Dou2021Moving} and thus is omitted. Specifically, the targets' geometrical center is equally represented as the leader's average position (LAP) in the following sections. Formally, it can be computed by:
\begin{equation} \label{eq. LAPdefinition}
	\vect{\bar{p}}^t = \frac{1}{m} \sum_{i \in \mathcal{V}_F} \vect{p}_i^t
\end{equation}
Given the LAP, we can define the relative position between follower $i$ and the LAP as:
$$\Delta \vect{p}_i = \vect{p}_i -\vect{\bar{p}^t} = [\Delta p_i^x,\Delta p_i^y]^T$$
Then the relative angle between the follower $i$ and the LAP in the global frame is be denoted by 
$$\theta_i = atan2(\frac{\Delta p_i^y}{\Delta p_i^x}).$$
As a result, the included angle from follower $i$ to its successor $i_{+}$ is defined by:
\begin{equation}
	\bar{\theta}_i = \theta_{i_{+}}-\theta_{i} + \varsigma_i
\end{equation}
where
\begin{equation}
	\varsigma_i=\left\lbrace \begin{aligned}
		0,\text{if } \theta_{i_{+}}-\theta_i \ge 0\\
		2\pi, \text{if }\theta_{i_{+}}-\theta_i < 0
	\end{aligned}  \right.
\end{equation}

A predefined spacing pattern is encoded by $\mathbb{P} = \{ w_d, \rho_d, \vect{a} \}$, which contains a set of desired included angles $\vect{a} = [a_1,a_2,\cdots,a_n]^T$, a desired circling radius $\rho_d$ and a desired angular velocity $w_d$ of followers when they finally rotate around the LAP.  The pattern is said to be admissible if and only if $\sum_{i=1}^n a_i = 2\pi, w_d > 0,$ and $\rho_d>max_{i \in \mathcal{V}_L}(||\vect{p}_i^t-\vect{\bar{p}}^t||_2)$.

\textit{\textbf{Problem 1}:} Given Assumptions~\ref{ass. undirectNetwork} - \ref{ass. boundedVelocityLeader} and the MAS composed of leaders in \eqref{eq. leadersDynamics} and followers in \eqref{eq. followersDynamics}, design a distributed controller for all followers such that they encircle the leaders (i.e. move along a common circle centered at the LAP) and the desired admissible spacing pattern $\mathbb{P}$ is achieved in finite time. 

%
%

\section{Continuous finite-time estimator}

Acquiring the exact LAP is the key for every follower to distributively accomplish the \textit{Problem 1}. One intuitive method is to collect all leaders' positions in each follower and then make the calculation directly according to its definition in \eqref{eq. LAPdefinition}. However, it is usually difficulty or impossible in reality due to various practical limitations. A more meaningful approach is to maintain an estimator for the LAP using only local information, for example, the observations to a subset of neighboring leaders and the exchanged messages from other neighboring followers. A well-studied estimator is designed as:
\begin{equation}
	\left \lbrace \begin{aligned}
		&\vect{{r}}_i = \vect{\Phi}_i+\vect{\tilde{p}}_i, \qquad \vect{\tilde{p}}_i=\frac{n}{m}\sum_{j\in \vect{N}_i^T}\frac{1}{|\vect{N}_j^F|}\vect{p}_j^t. \\	
		&\vect{\dot{\Phi}}_i=k_{sgn}\sum_{j\in \vect{N}_i}\frac{\vect{r}_j-\vect{r}_i}{|\vect{r}_j-\vect{r}_i|}
	\end{aligned} \right.
\end{equation}

It has been investigated in \cite{shi2015cooperative,sharghi2019finite,zhang2019finite,ma2017finite} that the output of this estimator $\vect{{r}}_i$ can exactly track the real LAP in finite time if Assumptions~\ref{ass. undirectNetwork} - \ref{ass. boundedVelocityLeader} hold and the estimator gain has $k_{sgn} > \beta (n-1)$.

However, such estimator consists of the discontinuous term $\frac{\vect{r}_j-\vect{r}_i}{|\vect{r}_j-\vect{r}_i|}$, which can result in dithering when either the estimator is getting converging or the adjacent estimators are similar. Regarding this problem, we propose a continuous estimator as following,
\begin{equation} \label{eq. smoothingEstimator}
	\left \lbrace
	\begin{aligned}
		&\vect{{r}}_i = \vect{\Phi}_i+\vect{\tilde{p}}_i, \qquad \vect{\tilde{p}}_i=\frac{n}{m}\sum_{j\in \vect{N}_i^T}\frac{1}{|\vect{N}_j^F|}\vect{p}_j. \\	
		&\vect{\dot{\Phi}}_i=k_e\sum_{j\in \vect{N}_i} a_{ij}\left( \vect{r}_j - \vect{r}_i \right)^{\frac{1}{\alpha_1}}
	\end{aligned} \right.
\end{equation}
where $\alpha_1 = \frac{p_1}{q_1}$, and $p_1>q_1>0$ are two positive odds. $k_e > 0$ is the estimator gain. To show its effectiveness, we present the following theorem, 

\begin{theorem} \label{Th. smoothingEstimator}
	Given Assumptions~\ref{ass. undirectNetwork} - \ref{ass. boundedVelocityLeader}, considering the system in \eqref{eq. leadersDynamics} and \eqref{eq. followersDynamics}, then if each follower in the MAS maintains an estimator \eqref{eq. smoothingEstimator}, there exist a settling time $T_1$ and a positive threshold $\epsilon$, such that,
	\begin{equation}
		  || \vect{r}_i - \vect{\bar{p}}^t ||_2 \le \epsilon, \forall t \ge T_1
	\end{equation}
	where the threshold can be adjusted by tuning the parameters in estimator.
\end{theorem}

\begin{proof}
Considering a Lyapunouv candidate $V_1  =  \frac{1}{2} \sum_{i,j=1}^n  || \vect{r}_i - \vect{r}_j ||_2^2 =  \frac{1}{2} \sum_{i,j=1}^n  \left( \vect{r}_i - \vect{r}_j \right)^T  \left( \vect{r}_i - \vect{r}_j \right)$, its derivative has 
\begin{equation}
	\begin{aligned}
			&\dot{V_1} = \frac{1}{2} \sum_{i,j=1}^n   \left[  (\vect{\dot{r}}_i - \vect{\dot{r}}_j)^T \left( \vect{r}_i - \vect{r}_j \right) + \left( \vect{r}_i - \vect{r}_j \right)^T (\vect{\dot{r}}_i - \vect{\dot{r}}_j) \right] \\
			& =   \sum_{i,j=1}^n  \left[  \left( \vect{r}_i - \vect{r}_j \right)^T (\vect{\dot{\Phi}}_i - \vect{\dot{\Phi}}_j) + \left( \vect{r}_i - \vect{r}_j \right)^T (\vect{\dot{\tilde{p}}}_i - \vect{\dot{\tilde{p}}}_j)  \right]  \\
\end{aligned}
\end{equation}
Substituting the estimator in \eqref{eq. smoothingEstimator} yields
\begin{equation}
	\begin{aligned}
	        &\dot{V} = \sum_{i,j=1}^n  \left[  -k_e\left( \vect{r}_i - \vect{r}_j \right)^T \sum_{s_1 =1}^n a_{is_1}\left( \vect{r}_i - \vect{r}_{s_1} \right)^{\frac{1}{\alpha_1}} \right.  \\
			& \left. -k_e\left( \vect{r}_i - \vect{r}_j \right)^T \sum_{s_2 =1}^n a_{js_2}\left( \vect{r}_{s_2} - \vect{r}_j \right)^{\frac{1}{\alpha_1}}  + \left( \vect{r}_i - \vect{r}_j \right)^T (\vect{\dot{\tilde{p}}}_i - \vect{\dot{\tilde{p}}}_j) \right]\\
			& = -\frac{k_e}{2} \sum_{i,j=1}^n    \left( \sum_{s_1 = 1}^n a_{is_1} (\vect{r}_i - \vect{r}_{s_1})^{\frac{1+\alpha_1}{\alpha_1}} +  \sum_{s_2 = 1}^n a_{js_2} (\vect{r}_j - \vect{r}_{s_2})^{\frac{1+\alpha_1}{\alpha_1}} \right) + \\
			& k_e \sum_{i=1}^n  \vect{r}_i \sum_{j,s_2=1}^n a_{js_2} (\vect{r}_j - \vect{r}_{s_2})^{\frac{1}{\alpha_1}} + k_e \sum_{j=1}^n \vect{r}_j \sum_{i,s_1=1}^n  a_{is_1} (\vect{r}_i - \vect{r}_{s_1})^{\frac{1}{\alpha_1}} \\
			& + \sum_{i,j=1} \left[ \left( \vect{r}_i - \vect{r}_j \right)^T (\vect{\dot{\tilde{p}}}_i - \vect{\dot{\tilde{p}}}_j) \right]
	\end{aligned}
\end{equation}
Recalling Lemma~\ref{lemma. sum_sum}, it's clear that
$$\sum_{j=1}^n  \sum_{s_2=1}^n a_{js_2} (\vect{r}_j - \vect{r}_{s_2})^{\frac{1}{\alpha_1}} = 0$$
and
$$\sum_{i=1}^n  \sum_{s_1=1}^n a_{is_1} (\vect{r}_i - \vect{r}_{s_1})^{\frac{1}{\alpha_1}} = 0$$
Therefore, the derivative can be rewritten into:
\begin{equation}
	\begin{aligned}
		\dot{V_1} &= -\frac{n k_e}{2} \left[ \sum_{i,s_1=1}^n  a_{is_1} \left( \vect{r}_i - \vect{r}_{s_1} \right)^{\frac{1+\alpha_1}{\alpha_1}} + \sum_{js_2=1}^n a_{js_2} \left( \vect{r}_j - \vect{r}_{s_2} \right)^{\frac{1+\alpha_1}{\alpha_1}}  \right] \\
		&+ \sum_{i,j=1} \left[ \left( \vect{r}_i - \vect{r}_j \right)^T (\vect{\dot{\tilde{p}}}_i - \vect{\dot{\tilde{p}}}_j) \right] \\
		& = - n k_e  \sum_{i,s_1=1}^n  a_{is_1} \left( \vect{r}_i - \vect{r}_{s_1} \right)^{1+\frac{1}{\alpha_1}} + \sum_{i,j=1} \left[ \left( \vect{r}_i - \vect{r}_j \right)^T (\vect{\dot{\tilde{p}}}_i - \vect{\dot{\tilde{p}}}_j) \right] \\
		& \le - n k_e  \sum_{i,s_1=1}^n  a_{is_1} \left( \vect{r}_i - \vect{r}_{s_1} \right)^{1+\frac{1}{\alpha_1}} + 2\sum_{i,j=1} \left[  || \vect{\dot{\tilde{p}}}_i ||_2^{\frac{1}{2}} || \vect{r}_i - \vect{r}_j ||_2  \right]
	\end{aligned}
\end{equation}
From the definition of $\tilde{\vect{p}}_i$ and Assumption~\ref{ass. leastObservation}, we know that $|\vect{N}_j^F| \ge 1$ and $|\vect{N}_i^T | \le m$, which yields
$$ \begin{aligned}
	|| \vect{\dot{\tilde{p}}}_i ||_2 &\le \frac{n}{m}\sum_{j\in \vect{N}_i^T}\frac{1}{|\vect{N}_j^F|}||\vect{\dot{p}}_j^t||_2 \\
	& \le  n \text{max}_{j \in \mathcal{V}_T} ||\vect{\dot{p}}_j^t||_2 = n \beta
\end{aligned} $$
As a result, 
\begin{equation}
	\begin{aligned}
		\dot{V_1} & \le - n k_e  \sum_{i,s_1=1}^n  a_{is_1} \left( \vect{r}_i - \vect{r}_{s_1} \right)^{1+\frac{1}{\alpha_1}} + 2 n \beta \sum_{i,j=1}   || \vect{r}_i - \vect{r}_j ||_2 \\
		& \le -n k_e V_1^{\frac{\alpha_1 + 1}{2\alpha_1}} + 2n\beta V_1^{\frac{1}{2}} .
	\end{aligned}
\end{equation}
If we let $$-n k_e V_1^{\frac{\alpha_1 + 1}{2\alpha_1}} + 2n\beta V_1^{\frac{1}{2}} \le - \eta V_1^{\frac{1}{2}} $$
i.e. 
$$ \begin{aligned}
	-n k_e V_1^{\frac{1}{2\alpha_1}} &+ 2n\beta \le - \eta \\
	V_1 &\ge \left( \frac{\eta +2n\beta}{nk_e} \right)^{2\alpha_1} := f(\beta,k_e,\alpha_1)
\end{aligned} $$
where $\eta > 0$ is an arbitrary and $f(\beta,k_e,\alpha_1)$ is a single-valued function with respect to the inputs, then $$\dot{V}_1 \le -\eta V_1^{\frac{1}{2}}$$
which means the Lyapunov function will converge to the set $[0,f(\beta,k_e,\alpha_1)]$ in finite time $$T_1 \le 2\frac{V_1^{\frac{1}{2}}(0)-f^{\frac{1}{2}}(\beta,k_e,\alpha_1)}{\eta}.$$ 
and will stay within the set, i.e. $V_1(t) \in [0,f(\beta,k_e,\alpha_1)], \forall t > T_1$. 

Since $V_1 = \sum_{i,j=1} (\vect{r}_i - \vect{r}_j)^2$, for any two followers $i,j$, the difference between their estimates has
$$||\vect{r}_i - \vect{r}_j|| \le \sqrt{V_1} \le f^{\frac{1}{2}}(\beta,k_e,\alpha_1) $$ 
before the settling time $T_1$. In addition, 
$$\begin{aligned}
	\sum_{i=1}^n \vect{r}_i(t) &= k_e\int_0^t  \sum_{i=1}^n \sum_{j\in \vect{N}_i} a_{ij}(\vect{r}_j - \vect{r}_i)^{\frac{1}{\alpha_1}} +  \sum_{i=1}^n \vect{\tilde{p}}_i \\
\end{aligned} $$
Since $$\sum_{i=1}^n \sum_{j\in \vect{N}_i} a_{ij}(\vect{r}_j - \vect{r}_i)^{\frac{1}{\alpha_1}} = 0$$
and $$\sum_{i=1}^n \vect{\tilde{p}}_i = \sum_{i=1}^n \frac{n}{m}\sum_{j\in \vect{N}_i^T}\frac{1}{|\vect{N}_j^F|}\vect{p}_j^t = \frac{n}{m} \sum_{i=1}^m \vect{p}^t_j = n\vect{\bar{p}}^t $$
Therefore, $$\frac{1}{n} \sum_{i=1}^n \vect{r}_i = \vect{\bar{p}}_t.$$ 
For any follower $i \in \mathcal{V}_F$, the error between the real LAP and the output of the estimate in \eqref{eq. smoothingEstimator} is
$$\begin{aligned}
	|| \vect{r}_i - \vect{\bar{p}}^t ||  &= || \vect{r}_i - \frac{1}{n} \sum_{j=1}^{n} \vect{r}_j ||\\
	& = || \frac{1}{n} \sum_{j=1}^n ( \vect{r}_i - \vect{r}_j )||\\
	& \le \frac{1}{n} \sum_{j=1}^n ||( \vect{r}_i - \vect{r}_j )|| \\
	& \le f^{\frac{1}{2}}(\beta,k_e,\alpha_1)
\end{aligned}$$

Let define $\epsilon = f^{\frac{1}{2}}(\beta,k_e,\alpha_1)$, then this ends the proof of Theorem~\ref{Th. smoothingEstimator}.
\end{proof}

Remark that Theorem~\ref{Th. smoothingEstimator} ensures that the proposed continuous estimator can  achieve a bounded accuracy in finite time. The upper bound of accuracy can be adjusted by tuning the parameters $k_e,\alpha_1$ in the estimator. In practice, the preferred principles for determining $k_e$ and  $\alpha_1$ could be:
\begin{enumerate}
	\item choose a large $k_e>0$ such that $nk_e > \eta + 2n\beta$. 
	\item given the desired accuracy $\epsilon_d$ and settling time $T_{1,d}$, $\alpha_1$ should satisfy:
	$$\alpha_1 > \frac{\ln \epsilon_d}{\ln (\eta + 2n\beta) - \ln (nk_e)}$$ 
	and
	$$ \alpha_1 > \frac{\ln\left[ V_1^{\frac{1}{2}}(0) - \frac{T_{1,d}n(k_e - 2\beta)}{2} \right]}{\ln (\eta + 2n\beta) - \ln (nk_e)} $$ 
\end{enumerate}

\section{Controller design}
In this section, a controller is designed for \textit{Problem 1} to form the desired encircling formation. Firstly, we will show the property of finite-time convergence in Subsection~\ref{subsec. finite-time-controller} for the proposed controller without considering the estimated errors of the continuous estimators in \eqref{eq. smoothingEstimator}. Then, the impact of estimated errors is analyzed in Subsection~\ref{subsec. bounded-error}.

\subsection{Finite-time enclosing controller} \label{subsec. finite-time-controller}
After obtaining the estimate of LAP in each follower, we can define the estimated relative distance error:
\begin{equation}
	z_i={\rho}_i^{'}-\rho_d
\end{equation}
where ${\rho}_i^{'} = ||\Delta \vect{p}_i^{'}||$ and $\Delta \vect{p}_i^{'} = \vect{p}_i - \vect{r}_i$ are the relative position vector between the follower and the estimated LAP. Similar to the definition of $\theta_i$, we can further define the estimated included angle using $\Delta \vect{p}_i^{'}$ as 
$$\theta_i^{'} = atan2(\frac{\Delta p_i^{'y}}{\Delta p_i^{'x}}). $$  
Then the error of estimated included angle is defined as:
\begin{equation}
	\delta_i= \hat{\theta} -1, \quad \hat{\theta} = \frac{\bar{\theta}_i^{'}}{a_i}, \quad \bar{\theta}_i^{'} = \theta_{i+}^{'} - \theta_{i}^{'} + \varsigma_i
\end{equation}
It's easy to verify that $\sum_{i=1}^{n} a_i\delta_i = \sum_{i=1}^n \bar{\theta}_i - \sum_{i=1}^n a_i = 2\pi - 2\pi = 0 $, i.e. $\vect{a}^T\vect{\delta} = 0$. The derivative of relative error is :
\begin{equation}
	\begin{aligned}
		\dot{z}_i &= \dot{\rho}_i^{'}\\
		&=\frac{d}{dt} \sqrt{(\Delta p_i^{'x})^2 +(\Delta p_i^{'y})^2} \\
		&=\frac{1}{2\rho_i^{'}}  \dot{\Delta p_i^{'x}} \Delta p_i^{'x} + \dot{\Delta p_i^{'y}} \Delta p_i^{'y} \\
		&=\frac{1}{2\rho_i^{'}}\Delta \vect{\dot{p}}_i^{'T} \Delta \vect{p}_i^{'} =\frac{1}{2\rho_i^{'}}\Delta \vect{{p}}_i^{'T} \Delta \vect{\dot{p}}_i{'} = \frac{1}{2}\vect{\varphi}_i^{'T} \Delta \vect{\dot{p}}_i^{'}
	\end{aligned}
\end{equation}
Regarding the included angle error, one has:
\begin{equation}
	\dot{\delta}_i = \frac{\dot{\bar{\theta}}_i^{'}}{a_i}-1, \quad \dot{\bar{\theta}}_i = \dot{\theta}_{i_{+}}^{'}-\dot{\theta}_{i}^{'}
\end{equation}
Further, we have:
\begin{equation}
	\begin{aligned}
		\dot{\theta}_{i}^{'} &= \frac{1}{1+(\frac{\Delta p_i^{'y}}{\Delta p_i^{'y}})^2} \frac{d}{dt}(\frac{\Delta p_i^{'y}}{\Delta p_i^{'y}})\\
		&=\frac{1}{\rho_i^{'}}[-\Delta p_i^{'y},\Delta p_i^{'x}][-\Delta \dot{p}_i^{'x},\Delta \dot{p}_i^{'y}]^T\frac{1}{\rho_i^{'}}\\
		&=\vect{\varphi}_{i\perp}^{'T}\Delta \vect{\dot{p}}_i^{'} \frac{1}{\rho_i^{'}}		
	\end{aligned}
\end{equation}
Let $\Delta \vect{p}^{'}=[\Delta \vect{p}_1^{'},\Delta \vect{p}_2^{'},\cdots,\Delta \vect{p}_n^{'}]^T, \vect{z}=[z_1,z_2,\cdots,z_n]^T, \vect{\delta} = [\delta_1,\delta_2,\cdots,\delta_n]^T$ be the compact vector of these variables respectively. Then we have 
\begin{equation} \label{eq. originalDynamicsErrors}
	\left \lbrace
	\begin{aligned}
		\dot{\vect{z}} &= \frac{1}{2} diag([\vect{\varphi}_i^{'T}]) \Delta \vect{\dot{p}}^{'}=\frac{1}{2}\vect{\phi} \Delta \vect{\dot{p}}^{'} \\
		\vect{\dot{\delta}} &=- \vect{D}^{-1} \vect{\Lambda}^{-1} \vect{L}_{\theta}\vect{\phi}_{\perp}\Delta\vect{\dot{p}}^{'}
	\end{aligned} \right.
\end{equation}
where $\vect{D}= diag([\rho_i^{'}]) \in \mathbb{R}^{n\times n} ,\vect{L}_{\theta}\in \mathbb{R}^{n\times n} = diag([l_{ij}])$ and $L_{i,i}=1,L_{i,i_{+}}=-1$ and other elements are all zeros. The newly introduced matrices are defined as $\vect{\Lambda} = diag(\vect{a}) \in\mathbb{R}^{n\times n},\vect{\phi} = diag([\vect{\varphi}_i^{'T}]) \in\mathbb{R}^{n\times 2n}$, and $\vect{\phi}_{\perp} = diag([\vect{\varphi}_{i\perp}^{'T}]) \in\mathbb{R}^{n\times 2n}$.

Then we design the following finite-time enclosing controller as:

\begin{equation} \label{eq. finite-time-controller-single}
	\left \lbrace
	\begin{aligned}
		& \vect{u}_i = \vect{\varphi}_i^T w_{z,i} + \vect{\varphi}_{i,\perp}^T w_{\delta , i} + \vect{\dot{r}}_i \\
		&w_{z,i} = - k_z z^{\frac{1}{\alpha_2}}_i \\
		&w_{\delta,i} = k_{\delta} \delta^{\frac{1}{\alpha_3}}_i + w_d		
	\end{aligned} \right. 
\end{equation}
It can be rewritten into a compact form as 
\begin{equation} \label{eq. finite-time-controller}
	\left \lbrace
	\begin{aligned}
		& \vect{u} = \vect{\phi}^T \vect{w}_z + \vect{\phi}_{\perp}^T \vect{w}_{\delta} + \vect{\dot{r}} \\
		&\vect{w}_z = - k_z \vect{z}^{\frac{1}{\alpha_2}} \\
		&\vect{w}_{\delta} = k_{\delta}\vect{\delta}^{\frac{1}{\alpha_3}} + \vect{w}_d
	\end{aligned} \right. 
\end{equation}
where $k_z,k_{\delta}$ are two positive control gains, $\alpha_2 = \frac{p_2}{q_2},\alpha_3 = \frac{p_3}{q_3}$, $p_2> q_2 > 0, p_3 > q_3 > 0$ are all positive odds, $\vect{\dot{r}}$ is the derivative of the continuous estimator in \eqref{eq. smoothingEstimator}, $w_d$ is a constant indicating the desired rotating velocity for the enclosing formation.

Observing that $\Delta \vect{\dot{p}^{'}} = \vect{u} - \vect{\dot{r}}$ and $\vect{L}_{\theta} \vect{w}_d = \vect{0}$, we can substitute the controller in \eqref{eq. finite-time-controller} into the error dynamics \eqref{eq. originalDynamicsErrors}, which yields
\begin{equation}
	\left \lbrace
	\begin{aligned}
		\dot{\vect{z}} &= - \frac{k_z}{2} \vect{z}^{\frac{1}{\alpha_2}} \\
		\dot{\vect{\delta}} &= - k_{\delta} \vect{D}^{-1} \vect{\Lambda}^{-1} \vect{L}_{\theta} \vect{\delta}^{\frac{1}{\alpha_3}}
	\end{aligned} \right.
\end{equation}

Firstly, we show that the dynamic of the estimated included angle error $\vect{\delta}$ is globally asymptotically stable by the following lemma,
\begin{lemma} \label{lemma. homogenous1}
	Let $\alpha_3 = \frac{p_3}{q_3}$ and $p_3 > q_3 > 1$ be two positive odds, then the origin of 
	$$\dot{\vect{\delta}} = - k_{\delta}\vect{D}^{-1} \vect{\Lambda}^{-1} \vect{L}_{\theta} \vect{\delta}^{\frac{1}{\alpha_3}}$$ 
	is globally asymptotically stable.
\end{lemma}
\begin{proof}
	Considering the Lyapunouv candidate function $V_{2} = \frac{\alpha_3+1}{\alpha_3}\vect{\delta}^T \vect{\delta}^{\frac{1}{\alpha_3}} = \frac{\alpha_3+1}{\alpha_3} \sum_{i=1}^n \delta_i^{1+\frac{1}{\alpha_3}}$, its derivative has:
	\begin{equation}
		\begin{aligned}
			\dot{V}_2 &= \sum_{i=1}^n \delta_i^{\frac{1}{\alpha_3}} \dot{\delta}_i = -\sum_{i=1}^n \delta_i^{\frac{1}{\alpha_3}} \frac{k_{\delta}}{a_i \rho_i} \sum_{j=1}^n a_{ij}( \delta_i^{\frac{1}{\alpha_3}} - \delta_j^{\frac{1}{\alpha_3}} ) \\
			&= -\frac{1}{2}\sum_{i=1}^n \frac{k_{\delta}}{a_i \rho_i} \sum_{j=1}^n a_{ij}( \delta_i^{\frac{1}{\alpha_3}} - \delta_j^{\frac{1}{\alpha_3}} )^2  \le 0
		\end{aligned}
	\end{equation}
	The equality holds if and only if $ \vect{\delta} = e\vect{1}$, where $e$ is a constant. Since the error vector has the relationship of $\vect{a}^T \vect{\delta} = \vect{0}$, then the only solution for the equality is $e=0$. Thus, we have $\dot{V}_2\le 0$ with $\dot{V}_2=0$ if and only if $\vect{\delta} = \vect{0}$. Therefore, the origin is globally asymptotically stable. 
\end{proof}

\begin{figure}[tbp]
	\centering
	\includegraphics[width=\linewidth]{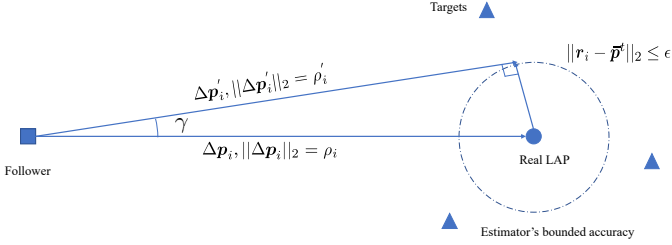}
	\caption{Illustration for the impact of bounded accuracy}
	\label{fig. estimatedErrorWithReal}
	\vspace{-15pt}
\end{figure}

Next, the fact that dynamic of the estimated included angle error $\vec{\delta}$ is a homogeneous system can be guaranteed by,
\begin{lemma} \label{lemma. homogenous2}
	Define the vector field $\vect{H}(\vect{\delta}) = - k_{\delta}\vect{D}^{-1} \vect{\Lambda}^{-1} \vect{L}_{\theta} \vect{\delta}^{\frac{1}{\alpha_3}}$, then it is homogeneous of degree $\frac{1}{\alpha_3} - 1$ with respect to the standard dilation. 
\end{lemma}
\begin{proof}
	Firstly, we show that $\delta_i^{\frac{1}{\alpha_3}}$ is homogeneous of degree $\frac{1}{\alpha_3}$ with respect to the standard dilation, i.e. for a $\lambda>0$:
	$$(\lambda \delta_i)^{\frac{1}{\alpha_3}} = \lambda ^{\frac{1}{\alpha_3}}  \delta_i^{\frac{1}{\alpha_3}}  $$
	Then, let's denote the $i$-th component of the vector field $\vect{H}(\vect{\delta})$ as $\vect{H}_i(\vect{\delta})= - \frac{1}{a_i \rho_i} \sum_{j=1}^n a_{ij}( \delta_i^{\frac{1}{\alpha_3}} - \delta_j^{\frac{1}{\alpha_3}} ) $, whereby we can get:
	$$\begin{aligned}
		\vect{H}_i(\lambda\vect{\delta}) &= - \frac{1}{a_i \rho_i} \sum_{j=1}^n a_{ij}\left[  (\lambda \delta_i)^{\frac{1}{\alpha_3}} - (\lambda \delta_j)^{\frac{1}{\alpha_3}} \right] \\
		& = - \lambda^{\frac{1}{\alpha_3}} \frac{1}{a_i \rho_i} \sum_{j=1}^n a_{ij}\left(  \delta_i^{\frac{1}{\alpha_3}} - \delta_j^{\frac{1}{\alpha_3}} \right) \\
		&  = \lambda^{\frac{1}{\alpha_3}} \vect{H}_i(\vect{\delta}) 
	\end{aligned}$$
	Therefore, we say that the vector filed is homogeneous of degree $\frac{1}{\alpha_3} - 1$ with respect to the standard dilation. 
\end{proof}

Finally, we can conclude our main results of the finite-time controller as 

\begin{theorem} \label{Th. estimatedControllerTheoremFiniteTime}
	Given Assumptions~\ref{ass. undirectNetwork} - \ref{ass. boundedVelocityLeader}, consider the MAS composed of $m$ leaders in \eqref{eq. leadersDynamics} and $n$ followers in \eqref{eq. followersDynamics}, then if the each follower maintains an estimator in \eqref{eq. smoothingEstimator} and take the control input as \eqref{eq. finite-time-controller-single}, then the origin of the estimated error dynamics in \eqref{eq. originalDynamicsErrors} are finite-time stable, i.e. there exist settling time $T_2,T_3$ such that:
	\begin{equation}
		\left \lbrace \begin{aligned}
		     | \rho_i^{'} - \rho_d | = 0, \forall t > T_2\\
			| \bar{\theta}_i^{'} - a_i | = 0, \forall t > T_3
		\end{aligned} \right.
	\end{equation}
\end{theorem}
\begin{proof}
	On the one hand, given the results from Lemma~\ref{lemma. homogenous1} - \ref{lemma. homogenous2}, we can immediately derive the property of finite time convergence for the estimated included angle errors $\vect{\delta}$ by recalling Theorem~\ref{Th. HomogeneousFiniteTimeStable}. Therefore, there exists a settling time $T_3$ such that $ | \bar{\theta}_i^{'} - a_i | = 0, \forall t>T_3$. 
	
	On the other hand, regarding the estimated relative distance errors $\vect{z}$, let consider the Lyapunouv function $V_2 = \frac{1}{2} \vect{z}^T\vect{z} = \frac{1}{2}\sum_{i=1}^{n} z_{i}^2$. According to lemma~\ref{lemma. SummationFraction},  its derivative has:
	\begin{equation}
		\begin{aligned}
			\dot{V}_2 &= -\sum_{i=1}^{n} z_{i}\dot{z}_{i} = -k_z \sum_{i=1}^{n} z_{i}^{1+\frac{1}{\alpha_2}} \\
			& = -k_z \sum_{i=1}^{n} \left( z_{i}^2 \right)^{\frac{\alpha_2+1}{2\alpha_2}} \le  - k_z \left( \sum_{i=1}^n z_{i}^2 \right)^{\frac{\alpha_2+1}{2\alpha_2}}  = - k_z V_2^{\frac{\alpha_2+1}{2\alpha_2}} 
		\end{aligned}
	\end{equation}
	Therefore, the estimated relative distance errors $\vect{z}$ will converge to zeros in finite time $T_2 \le \frac{2\alpha_2}{k_z(\alpha_2 - 1)} V_2(0)^{\frac{\alpha_2 -1 }{2\alpha_2}}$. 
	
	This ends the proof of Theorem~\ref{Th. estimatedControllerTheoremFiniteTime}.
\end{proof}

\begin{figure}[tbp]
	\centering
	\includegraphics[width=0.6\linewidth]{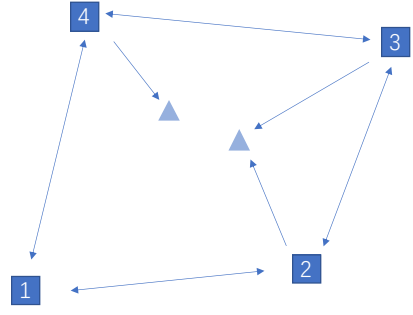}
	\caption{The network topology in the simulation case, where the rectangles and squares denote the followers and leaders, respectively. The arrows indicate connections within MAS. Note that the topology among followers is undirect and each leader is observed by at least one follower.}
	\label{fig. simulationTopology}
	\vspace{-15pt}
\end{figure}

\subsection{Bounded error} \label{subsec. bounded-error}

The previous subsection shows that the proposed controller in \eqref{eq. finite-time-controller-single} can achieve the desired spacing patter $\mathbb{P}$ in finite time if the estimator can exactly track the real LAP in finite-time. However, the estimator designed in \eqref{eq. smoothingEstimator} can only guarantee a bounded accuracy in finite time. Therefore, the impact of the bounded accuracy on the controller is discussed in this subsection. Let define the real errors of both relative distance and the included angle as 
\begin{equation}
	\left \lbrace
	\begin{aligned}
		e_{\rho,i} = \rho_i - \rho_d \\
		e_{\delta,i} = \frac{\bar{\theta}_i}{a_i} -1, 
	\end{aligned} \right.
\end{equation} 

\begin{figure*}[tbp]
	\centering
	\begin{subfigure}[b]{0.49\textwidth}
		\begin{subfigure}[b]{1\textwidth}
			\centering
			\includegraphics[width=10cm]{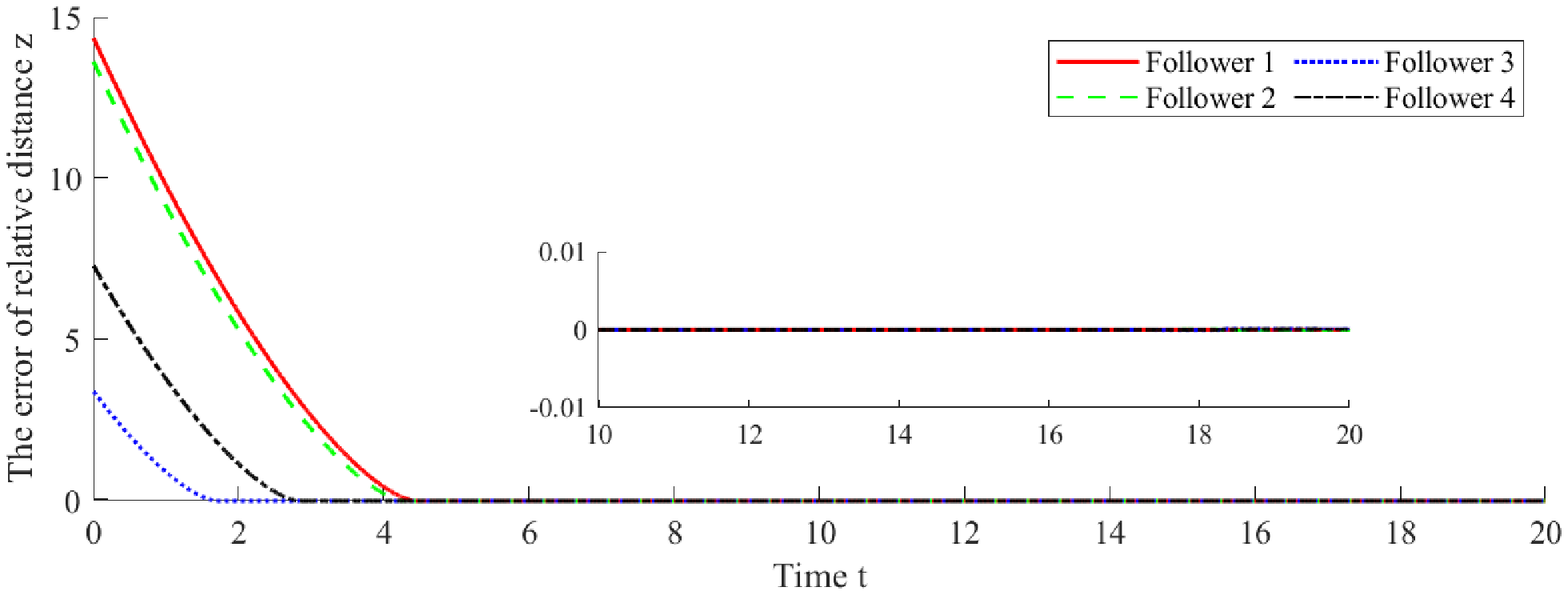}
			\caption{The error of estimated relative distance}
			\label{fig. z}
		\end{subfigure}
		\null \hfill \\
		\begin{subfigure}[b]{1\textwidth}
			\centering
			\includegraphics[width=10cm]{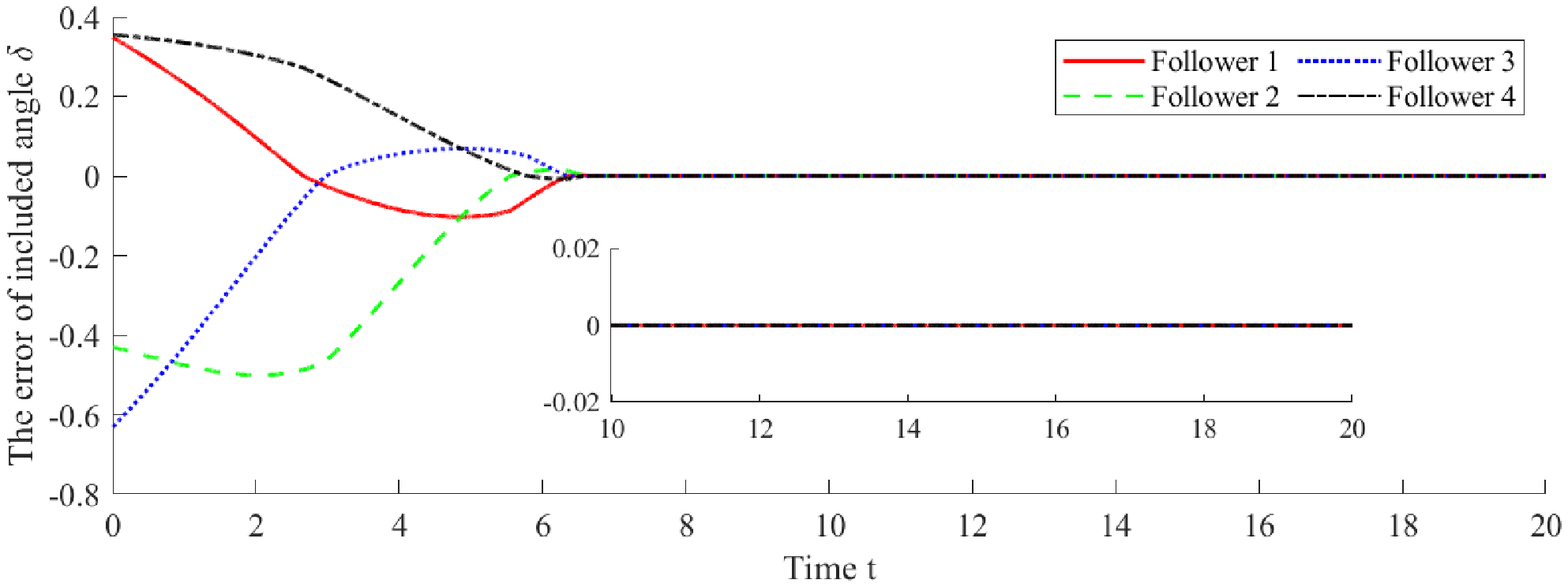}
			\caption{The error of estimated included angle}
			\label{fig. delta}
		\end{subfigure}
		\null \hfill \\
		\begin{subfigure}[b]{1\textwidth}
			\centering
			\includegraphics[width=10cm]{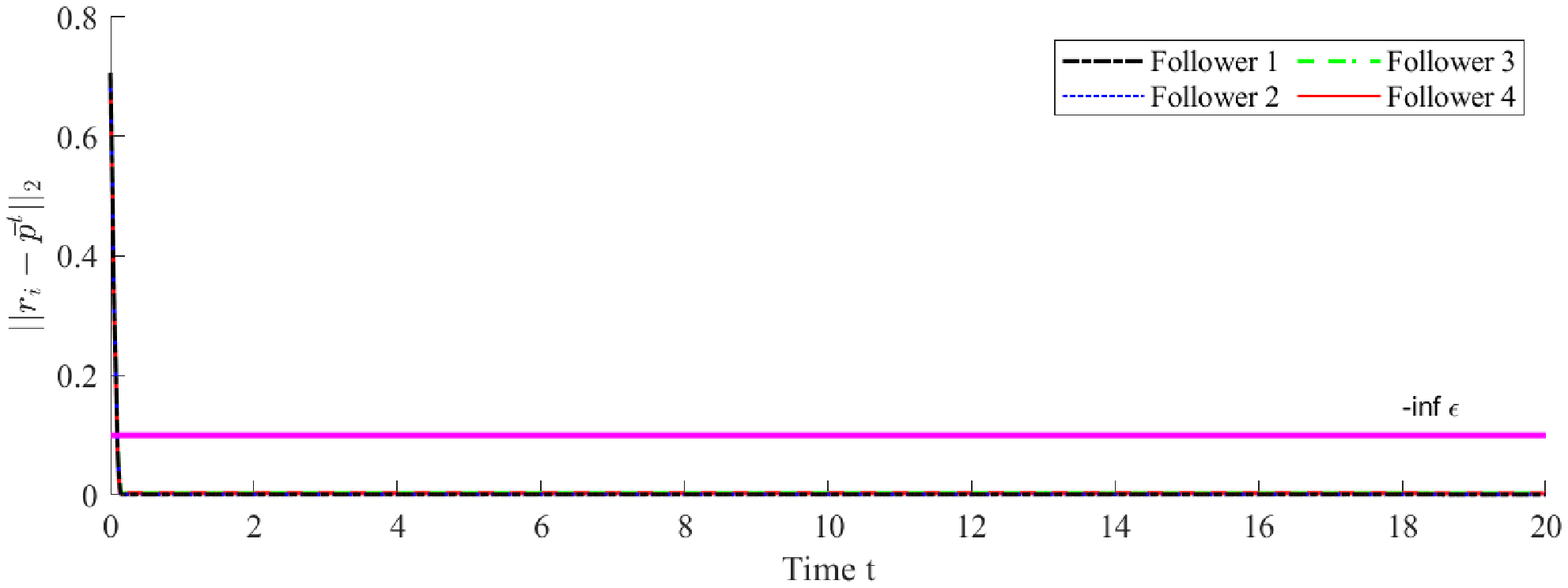}
			\caption{The estimated error between real LAP and the output of estimator }
			\label{fig. estimatorError}
		\end{subfigure}
	\end{subfigure}
	\hfill
	\begin{subfigure}[b]{0.49\textwidth}
		\begin{subfigure}[b]{1\textwidth}
			\centering
			\includegraphics[width=8.5cm]{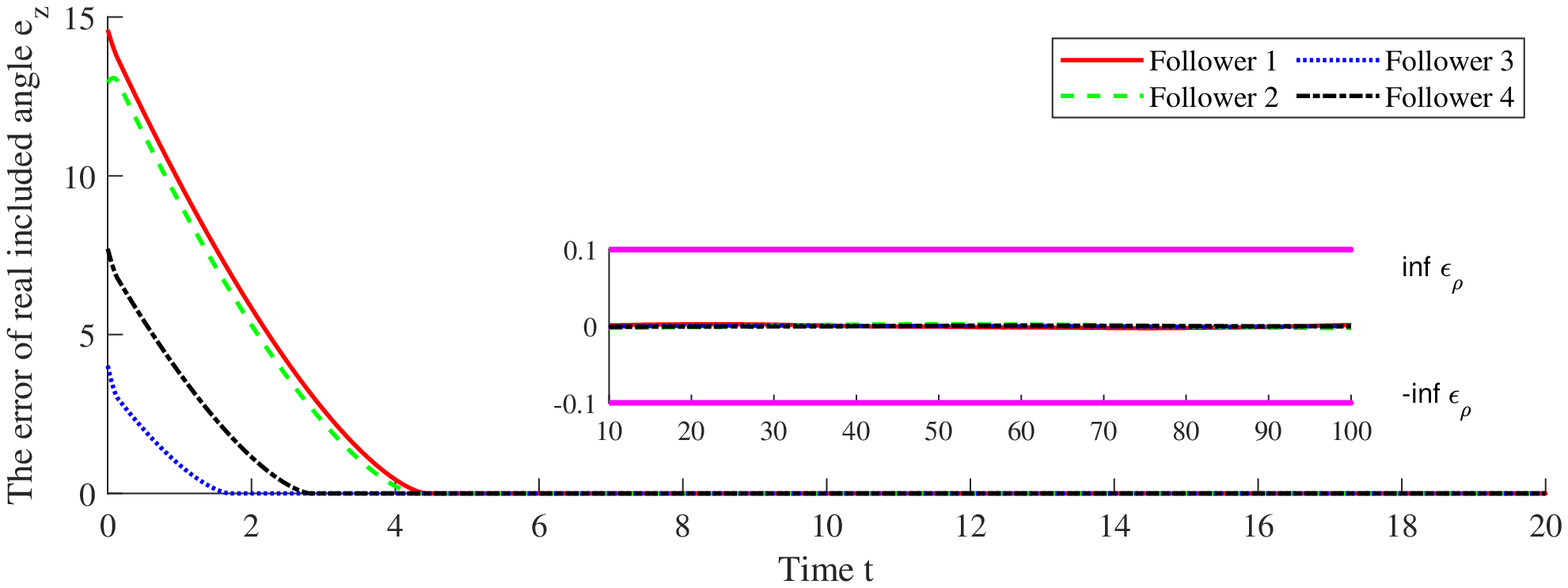}
			\caption{The error of real relative distance}
			\label{fig. e_z}
		\end{subfigure}
		\null \hfill \\
		\begin{subfigure}[b]{1\textwidth}
			\centering
			\includegraphics[width=8.5cm]{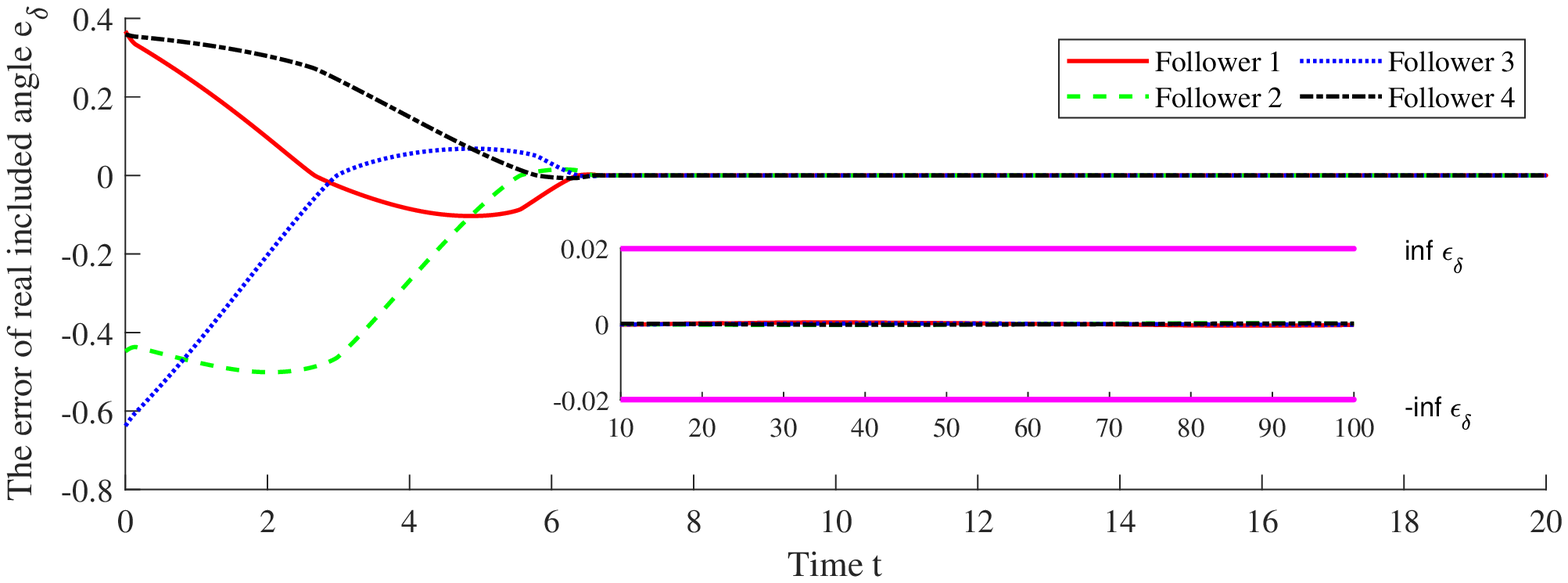}
			\caption{The error of real included angle}
			\label{fig. e_delta}
		\end{subfigure}
		\null \hfill \\
		\begin{subfigure}[b]{1\textwidth}
			\centering
			\includegraphics[width=8.5cm]{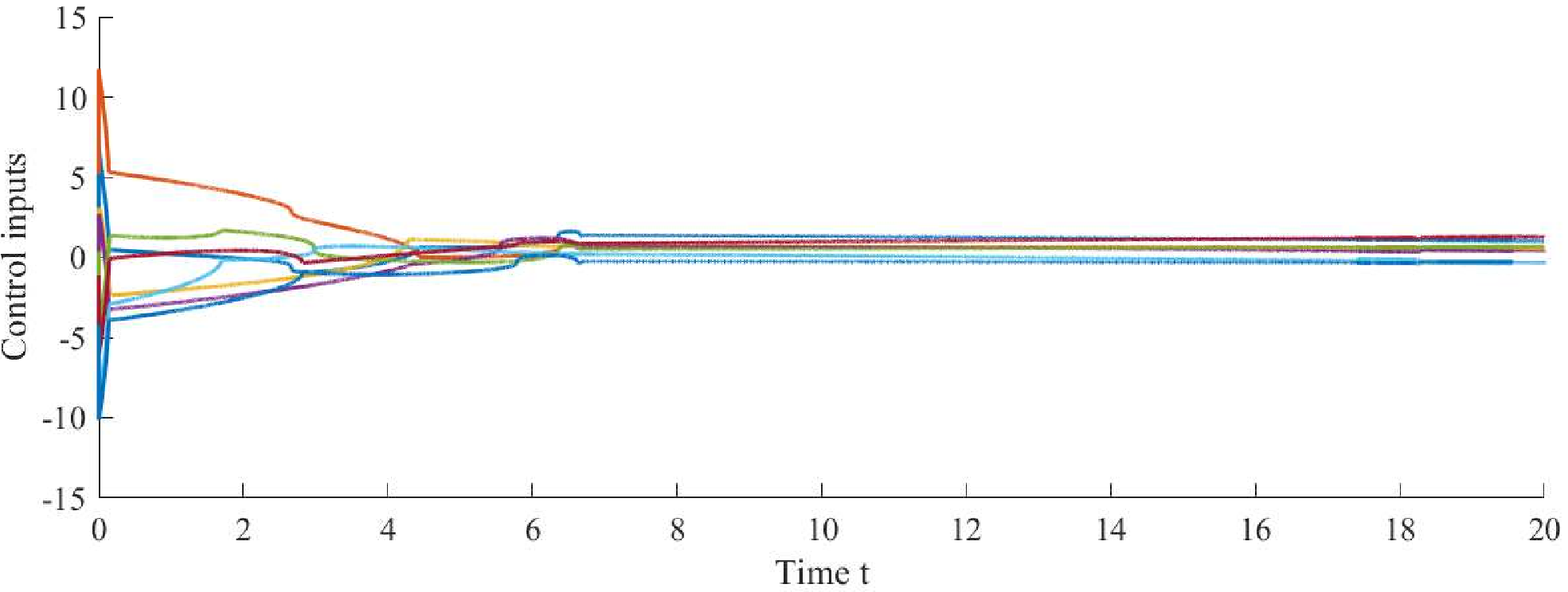}
			\caption{The control commands for all followers}
			\label{fig. controlCommands}
		\end{subfigure}
	\end{subfigure}
	\caption{The errors of both relative distances and included angles }
	\label{fig. performance}
	\vspace{-15pt}
\end{figure*}

Then we present our analysis as the following theorem,
\begin{theorem}\label{Th. errorRealEstimator}
	Given Assumptions~\ref{ass. undirectNetwork} - \ref{ass. boundedVelocityLeader}, consider the MAS composed of $m$ leaders in \eqref{eq. leadersDynamics} and $n$ followers in \eqref{eq. followersDynamics}, then if the each follower maintains an estimator in \eqref{eq. smoothingEstimator} and take the control input as \eqref{eq. finite-time-controller-single}, then there exist positive constants $\epsilon_{\rho}$ and $\epsilon_{\delta}$, such that 
	\begin{equation}
		\left \lbrace \begin{aligned}
			 |e_{\rho,i}| < \epsilon_{\rho},\forall t > T_1+T_2+T_3 \\
			 |e_{\delta,i}| < \epsilon_{\delta}, \forall t > T_1+T_2+T_3 \\
		\end{aligned} \right.
	\end{equation}
\end{theorem}
\begin{proof}
	Firstly, it's obvious that
	$$e_{\rho,i} = z_i + \rho_i - \rho_i^{'}$$
	and 
	$$
	\begin{aligned}
		e_{\delta,i} &= \delta_i + \frac{\bar{\theta}_i - \bar{\theta}^{'}_i}{a_i} \\
		&= \delta_i + \frac{ \left( \theta_{i+} - \theta_{i+}^{'} \right) - \left(   {\theta}_i - {\theta}^{'}_i \right) }{a_i} \\
		&= \delta_i + \frac{ \gamma_{i+} - \gamma_i }{a_i} 
	\end{aligned}
	$$
	where $\gamma_i = \theta_{i} - \theta_{i}^{'}$ is the error between the real included angle and the estimated one as depicted in Fig.~\ref{fig. estimatedErrorWithReal}. Then, we can derive,
	\begin{equation}
		|e_{\rho,i}| < |\rho_i - \rho_i^{'}| < \epsilon,\forall t > T_1+T_2+T_3
	\end{equation}
	and 
	\begin{equation}
		\begin{aligned}
			& \tan(\gamma_i) = \frac{\epsilon}{\rho_d}\\
			\Rightarrow &  |e_{\delta,i}| = |\frac{ \gamma_{i+} - \gamma_i }{a_i}| \le 2\frac{ \max |\gamma_i| }{a_i} < \frac{2 \arctan \frac{\epsilon}{\rho_d}}{a_i}
		\end{aligned}	
	\end{equation}
	$\forall t > T_1+T_2+T_3$. Therefore, we can define $\epsilon_{\rho} = \epsilon$ and $\epsilon_{\delta} = \frac{2 \arctan \frac{\epsilon}{\rho_d}}{a_i}$.
	
	This ends the proof of Theorem~\ref{Th. errorRealEstimator}.
\end{proof}

\section{Simulation}

\begin{figure*}[tbp]
	\centering
	\includegraphics[width=\linewidth]{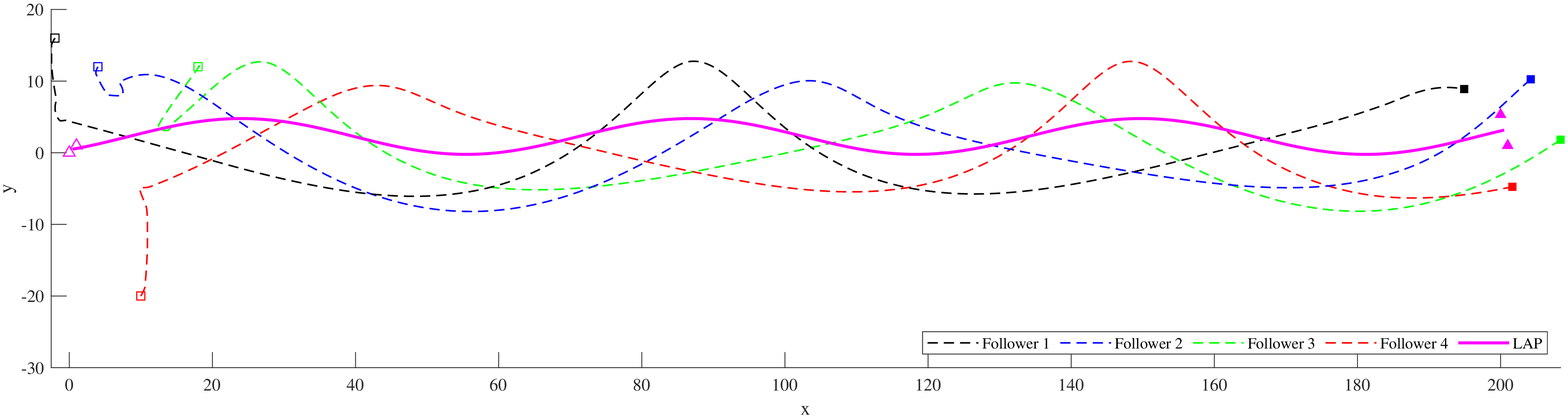}
	\caption{Trajectories of all four followers and the real LAP over 200s}
	\label{fig. trajectory}
	\vspace{-15pt}
\end{figure*}

In the simulation, we consider the case with four followers and two leaders moving in 2-D plane. The network topology is demonstrated in Fig.~\ref{fig. simulationTopology}, which meets the requirements of the Assumption.~\ref{ass. undirectNetwork} and ~\ref{ass. leastObservation}. The leaders are moving across the plane with bounded velocities that are unknown to all followers, which are respectively,
$$
\begin{aligned}
	\vect{v}_1 = \begin{bmatrix}
		1 \\
		0
	\end{bmatrix} , \quad
	\vect{v}_2 = \begin{bmatrix}
		1\\
		\frac{1}{2}\sin(\frac{t}{10}+ \frac{\pi}{4}) 
	\end{bmatrix}. 
\end{aligned}
$$
Therefore, the upper bound of target velocities is $\beta = \frac{\sqrt{5}}{2}$. The initial positions for both leaders and followers are listed in the Tab.~\ref{Tab. initialPositions}. 

The desired spacing pattern $\mathbb{P}$ is configured as follows. The desired relative distance between each agent and the real LAP is set to $\rho_d = 8$ m. The desired included angles between any pair of consecutive agents are fixed by $\vect{a} = [\frac{2\pi}{5},\frac{2\pi}{5},\frac{2\pi}{5},\frac{4\pi}{5}]^T$. The controller parameters are designed as in Tab.~\ref{Tab. ControllerParameters}. Through simple calculation, we derive that
$$ \left \lbrace
\begin{aligned}
	&\inf_{\eta \rightarrow 0} ( \epsilon, \epsilon_{\rho} ) = \left( \frac{2n\beta}{nk_e} \right)^{2\alpha_1} \approx  0.099949 \\
	&\inf_{\eta \rightarrow 0} \epsilon_{\delta} = \frac{2 \arctan \frac{\epsilon}{\rho_d}}{\min a_i} \approx 0.0199
\end{aligned} \right.
$$

\begin{table}[htbp]	
	\caption{Initial positions for the simulation case}
	\label{Tab. initialPositions}
	\centering
	\begin{tabular}{ccccccc}
		\hline
		\multirow{2}{*}{Nodes} & \multicolumn{4}{c}{Followers}                            & \multicolumn{2}{c}{Leaders} \\ \cline{2-7} 
		& $F_1$        & $F_1$        & $F_1$      & $F_1$         & $L_1$        & $L_2$        \\ 
		Pos(m)              & $[10;-20]$ & $[18;12]$ & $[4;12]$ & $[-2;16]$ & $[0;0]$   & $[1;1]$   \\ \hline
	\end{tabular}	
\end{table}

\begin{table}[htbp]
	\centering
	\caption{Parameters for the designed controller}
	\label{Tab. ControllerParameters}
	\begin{tabular}{llllll}
		\hline
		\multicolumn{1}{c}{$k_e$} & $\alpha_1$    & $k_z$ & $\alpha_2$    & $k_{\delta}$ & $\alpha_3$    \\ 
		6                         & $\frac{7}{3}$ & 2     & $3$ & 2            & $3$ \\ \hline
	\end{tabular}
	\vspace{-5pt}
\end{table}

The simulation results are demonstrated from  Figs.~\ref{fig. performance} - \ref{fig. trajectory}.  
Specifically, Fig.~\ref{fig. performance} shows the performance of the controller in \eqref{eq. finite-time-controller} and estimator in \eqref{eq. smoothingEstimator}, where Fig.~\ref{fig. z} and Fig.~\ref{fig. delta} are the evolution of the estimated errors of both the relative distances $z_i,\forall i \in \mathcal{V}_F$ and included angles $\delta_i,\forall i \in \mathcal{V}_F$. As proven by Theorem~\ref{Th. estimatedControllerTheoremFiniteTime}, they are exactly stabilized in finite-time. The real errors $e_{z,i}$ and $e_{\delta,i}$ demonstrated in Fig~\ref{fig. e_z} and Fig.~\ref{fig. e_delta}, however, converge to their theoretical accuracy in finite time as shown in Theorem~\ref{Th. errorRealEstimator}. 
Fig.~\ref{fig. estimatorError} shows the error between the real LAP and the output of continuous estimators that are maintained by all followers. We can see that the errors of all four estimators quickly converge into the stable region $[0,\epsilon]$, which validates the correctness of Theorem~\ref{Th. smoothingEstimator}. Although these errors are not exactly zeros in finite time, they become quite smaller than the theoretical accuracy $\epsilon$.
Fig.~\ref{fig. controlCommands} presents the control commands of all followers over the simulation, which validate that the designed controller and estimator don't cause dithering during the enclosing process. Lastly, Fig.~\ref{fig. trajectory} gives the trajectories of all followers up to 200s. From Figs.~\ref{fig. performance} - \ref{fig. trajectory}, we can see that the desired spacing pattern is achieved in finite time and can be maintained within the theoretical accuracy. 

\section{Conclusion}
The finite-time enclosing control problem using a multi-agent system for multiple moving targets is investigated in this paper. The proposed distributed estimator can track the geometrical center of multiple moving targets in finite time using only local information and the tracking accuracy is shown to be bounded and adjustable. Based on the output of this estimator, a pair of decentralized enclosing control laws are developed to achieve the desired spacing pattern in finite time. The steady errors of both the relative distance and included angle is bounded and can be adjusted by tuning the parameters in both estimator and controller. Future works may focus on reducing the tracking and controlling errors, extending the results to more general vehicles and so on.


%

\appendices



\ifCLASSOPTIONcaptionsoff
  \newpage
\fi



%
%
%
\bibliographystyle{IEEEtran}
\bibliography{IEEEabrv,mybibfile}

\end{document}